\documentclass[12pt]{article}
\usepackage{amsmath,amssymb,amsthm}
\usepackage{enumerate}
\usepackage{graphicx}
\usepackage[T1]{fontenc}
\usepackage{cite}

\def\draw #1 by #2 (#3){
	\vbox to #2{
		\hrule width #1 height 0pt depth 0pt
		\vfill
		\special{picture #3} 
	}
}

\def\scaleddraw #1 by #2 (#3 scaled #4){{
		\dimen0=#1 \dimen1=#2
		\divide\dimen0 by 1000 \multiply\dimen0 by #4
		\divide\dimen1 by 1000 \multiply\dimen1 by #4
		\draw \dimen0 by \dimen1 (#3 scaled #4)}
}

\newtheorem{theorem}{Theorem}[section]
\newtheorem{example}[theorem]{Example}
\newtheorem{problem}[theorem]{Problem}
\newtheorem{defin}[theorem]{Definition}
\newtheorem{lemma}[theorem]{Lemma}
\newtheorem{corollary}[theorem]{Corollary}
\newtheorem{remark}[theorem]{Remark}
\usepackage{xcolor}
\newtheorem{nt}{Note}

\newcommand{\singlespacing}{\let\CS=\@currsize\renewcommand{\baselinestretch}{1}\tiny\CS}
\newcommand{\oneandahalfspacing}{\let\CS=\@currsize\renewcommand{\baselinestretch}{1.25}\tiny\CS}
\newcommand{\doublespacing}{\let\CS=\@currsize\renewcommand{\baselinestretch}{1.35}\tiny\CS}

\newtheorem{conjecture}[theorem]{Conjecture}

\newtheorem{rule-def}[theorem]{Rule}

\begin{document}
	\baselineskip 16pt
	\newcommand{\la}{\lambda}
	\newcommand{\si}{\sigma}
	\newcommand{\ol}{1-\lambda}
	\newcommand{\be}{\begin{equation}}
		\newcommand{\ee}{\end{equation}}
	\newcommand{\bea}{\begin{eqnarray}}
		\newcommand{\eea}{\end{eqnarray}}

	\baselineskip=0.30in

	\begin{center}
		{\Large \bf On energy of graphs with self loops}  \\
		\vspace{8mm}
		
		{\large \bf B. R. Rakshith$^a$, Kinkar Chandra Das$^{b,}\footnote{Corresponding author}$, B. J. Manjunatha$^{c,d}$}
		
		\vspace{6mm}
		
		\baselineskip=0.20in
		
		{\it $^a$Department of Mathematics, Manipal Institute of Technology, \\
			Manipal Academy of Higher Education,
			Manipal 576 104, India.\/} \\
		{\rm E-mail:} {\tt ranmsc08@yahoo.co.in}\\[2mm]
		
		{\it $^b$Department of Mathematics, Sungkyunkwan University, \\
			Suwon 16419, Republic of Korea.\/} \\
		{\rm E-mail:} {\tt kinkardas2003@gmail.com}\\[2mm]

			{\it $^c$Department of Mathematics, Sri Jayachamarajendra College of Engineering, JSS Science and Technology University,
		Mysuru–570 006, India\/}\\[2mm]

		{\it $^d$Department of Mathematics, Vidyavardhaka College of Engineering\\
		Mysuru-570 002\\Affiliated to Visvesvaraya Technological University, Belagavi-590 018, India\/} \\
	{\rm E-mail:} {\tt manjubj@sjce.ac.in}\\
		
		\vspace{5mm}
		
	\end{center}

\vspace{5mm}

\baselineskip=0.20in

\begin{abstract}
Let $G$ be a simple graph on $n$ vertices with vertex set $V(G)$. The energy of $G$, denoted by, $\mathcal{E}(G)$ is the sum of all absolute values of the eigenvalues of the adjacency matrix $A(G)$. It is the first eigenvalue-based topological molecular index and is related to the molecular orbital
energy levels of $\pi$-electrons in conjugated hydrocarbons. Recently, the concept of energy of a graph is extended to a self-loop graph.  Let $S$ be a subset of $V(G)$. The graph $G_{S}$ is obtained from the graph $G$ by attaching a self-loop at each of the vertices of $G$ which are in the set $S$. The energy of the self-loop graph $G_{S}$, denoted by $\mathcal{E}(G_{S})$, is the sum of all absolute eigenvalues of the matrix $A(G_{S})$. Two non-isomorphic self-loop graphs are equienergetic if their energies are equal. Akbari et al. (2023) \cite{akbari} conjectured that there exist a subset $S$ of $V(G)$ such that  $\mathcal{E}(G_{S})> \mathcal{E}(G)$. In this paper, we confirm this conjecture. Also, we construct pairs of equienergetic self-loop graphs of order $24n$ for all  $n\ge 1$.
\bigskip

\noindent
{\bf AMS Classification:} 05C50.\\

\noindent
{\bf Key Words:} Energy of a graph, Singular value of a matrix, Self-loop graph.

\end{abstract}

\baselineskip=0.30in
\section{Introduction}
Let $G$ be a simple graph on $n$ vertices with vertex set $V(G)$. The adjacency matrix of $G$ is denoted by $A(G)$ and its $ij$th entry is equal to 1 if the vertices $v_{i}$ and $v_{j}$ are adjacent in $G$ and 0 otherwise. Let $\lambda_{1}\ge \lambda_{2}\ge \cdots\ge \lambda_{n}$ be the eigenvalues of $A(G)$. The energy of the graph $G$, denoted by $\mathcal{E}(G)$, is the sum of all absolute eigenvalues of $G$ ($A(G)$), i.e. $\mathcal{E}(G)=\sum_{i=1}^{n}|\lambda_{i}|$. For  recent studies including its chemical significance, see \cite{AD1,AD2,DA1,E1,E2,E3,A1,A2}. Let $S$ be a subset of $V(G)$. The graph $G_{S}$ is obtained from the graph $G$ by attaching a self-loop at each of the vertices of $G$ which are in the set $S$. The adjacency matrix of the self-loop graph $G_{S}$ is denoted by $A(G_{S}) =D_{S}(G)+A(G)$, where $D_{S}(G)$ is the diagonal matrix with its $i$th diagonal entry equal to 1 if the vertex $v_{i}$ has a self-loop in $G_{S}$, 0 otherwise. Let $\lambda^{S}_{1}\ge \lambda^{S}_{2}\ge\cdots\ge\lambda^{S}_{n}$ be the eigenvalues of $A(G_{S})$ and let $|S|=\alpha$. The energy of self-loop graph $G_{S}$ is denoted by $\mathcal{E}(G_{S})$ and is defined as \cite{gutman} $$\mathcal{E}(G_{S})=\sum_{i=1}^{n}\left|\lambda^{S}_{i}-\dfrac{\alpha}{n}\right|.$$
Two non-isomorphic self-loop graphs are equienergetic if their energies are equal. The concept of energy of a self-loop graph was put forward very recently by Gutman et al. in \cite{gutman}. The authors in \cite{gutman} proved that if $G$ is a bipartite graph, then $\mathcal{E}(G_{S})=\mathcal{E}(G_{V(G)\backslash S})$. Also, they conjectured that $\mathcal{E}(G_{S})>\mathcal{E}(G)$ for $1<|S|<n-1$. However this conjecture was disproved in \cite{c2}. More details on the energy and spectrum of $G_{S}$ can be found in \cite{akbari,gutman,c2}. Recently, in \cite{akbari}, the authors proved that for a bipartite graph $\mathcal{E}(G_{S})\ge\mathcal{E}(G)$, $1<|S|<n-1$, and gave a conjecture as follows:
\begin{conjecture}\label{c1}
	For every simple graph $G$ of order at least 2, there exists $S\subseteq V(G)$ such that $\mathcal{E}(G_S)>
	\mathcal{E}(G)$.
\end{conjecture}

In this paper we prove that the Conjecture \ref{c1} is true. Also, we construct pairs of equienergetic self-loop graph of order $24n$ for all  $n\ge 1$.

\section{Main results}
Let $A$ be an $n\times n$ complex matrix and denote its singular values by $s_{1}(A)\ge s_{2}(A)\ge\cdots\ge s_{n}(A)$. Since the matrices $A(G)$ and $A(G_{S})$ are symmetric, $\mathcal{E}(G)=\displaystyle\sum_{i=1}^{n}s_{i}(A(G))$ and $\mathcal{E}(G_S)=\displaystyle\sum_{i=1}^{n}s_{i}\left(A(G_{S})-\dfrac{\alpha}{n}\right)$. We need the following two lemmas to prove our main result.
\begin{lemma}{\rm\cite{so}}\label{l1}
Let $A$ and $B$ be two $n\times n$ matrices. Then
$$\sum_{i=1}^{n}s_{i}(A+B)\le\sum_{i=1}^{n}s_{i}(A)+\sum_{i=1}^{n}s_{i}(B)$$ and the equality holds if and only if there exists an unitary matrix $P$ such that both the matrices $PA$ and $PB$ are positive semi-definite. Moreover, if $A$ and $B$ are real matrices, then the matrix $P$ can be taken as real orthogonal.     	
\end{lemma}
\begin{lemma}{\rm\cite{horn}}\label{l2}
Let $A=[a_{ij}]_{n\times n}$ be a positive semi-definite matrix such that $a_{ii}=0$ for some $1\le i\le n$. Then $a_{ij}=a_{ji}=0$ for all $1\le j\le n$. 	
\end{lemma}

Let $G=(V,E)$ be any graph, and let $S\subseteq V(G)$ be any subset of vertices of $G$. Then the induced subgraph
$G[S]$ is the graph whose vertex set is $S$ and whose edge set consists of all of the edges in $E(G)$ that have both endpoints in $S$.
\begin{theorem}
Let $G$ be a graph and let $\emptyset\neq S\subsetneq V(G)$.
\begin{enumerate}[(i)]
	\item
If $\mathcal{E}(G_{S})<\mathcal{E}(G)$, then $\mathcal{E}(G_{V(G)\backslash S})>\mathcal{E}(G)$. Whereas if $\mathcal{E}(G_{S})=\mathcal{E}(G)$, then $\mathcal{E}(G_{V(G)\backslash S})\ge\mathcal{E}(G)$.
\item
If $S$ is a vertex independent set of a connected component $H$ (say) of $G$ such that $|V(H)|\ge 2$, then either $\mathcal{E}(G_{S})>\mathcal{E}(G)$ or $\mathcal{E}(V(G)\backslash S)>\mathcal{E}(G)$.       	
\end{enumerate}
\end{theorem}\label{t1}

\begin{proof}
Let $G[S]$ and $G[V(G)\backslash S]$ be the subgraph of $G$ induced by the vertex sets $S$ and $V(G)\backslash S$, respectively. Let $|S|=\alpha$.   Then $$A(G_{S})=\left[\begin{array}{cc}
A(G[S])+I_{\alpha}&X\\[3mm]
X^{T}& A(G[V(G)\backslash S])
\end{array}\right]$$
and $$A(G_{V(G)\backslash S})=\left[\begin{array}{cc}
A(G[S])&X\\[3mm]
X^{T}& A(G[V(S)\backslash S])+I_{n-\alpha}
\end{array}\right].$$ 

\vspace*{3mm}

\noindent
Therefore,  
\begin{equation}\label{eq1}
\left[A(G_{S})-\dfrac{\alpha}{n}I_{n}\right]+\left[A(G_{V(G)\backslash S})-\Big(\dfrac{n-\alpha}{n}\Big)\,I_{n}\right]=2A(G).
\end{equation}

\vspace*{3mm}

\noindent
Applying  Lemma \ref{l1} to the equation (\ref{eq1}), we get \begin{equation}\label{eq2}2\sum_{i=1}^{n}s_{i}(A(G))\le\sum_{i=1}^{n}s_{i}\left(A(G_{S})-\dfrac{\alpha}{n}I_{n}\right)+\sum_{i=1}^{n}s_{i}\left(A(G_{V(G)\backslash S})-\Big(\dfrac{n-\alpha}{n}\Big)\,I_{n}\right).
\end{equation}
That is,
\begin{equation}\label{eq3}
2\,\mathcal{E}(G)\le \mathcal{E}(G_{S})+\mathcal{E}(G_{V(G)\backslash S}).	
\end{equation}
If  $\mathcal{E}(G_{S})< \mathcal{E}(G)$, then by equation (\ref{eq3}), we get
$\mathcal{E}(G_{V(G)\backslash S})>	\mathcal{E}(G)$, and if  $\mathcal{E}(G_{S})= \mathcal{E}(G)$, then again by equation (\ref{eq3}), we get
$\mathcal{E}(G_{V(G)\backslash S})\ge \mathcal{E}(G)$. Proving (i).

\vspace*{3mm}

Let $S$ be a vertex independent set of a connected component $H$ of $G$ with $|V(H)|\ge 2$. Then 
$$A(G)=\left[\begin{array}{cc}
	\textbf{0}&X\\[3mm]
	X^{T}& A(G[V(G)\backslash S])
\end{array}\right]~~\mbox{ and }~~A(G_{S})=\left[\begin{array}{cc}
	I_{\alpha}&X\\[3mm]
	X^{T}& A(G[V(G)\backslash S])
\end{array}\right].$$ 
Since $|V(H)|\ge 2$, we must have $X\neq \textbf{0}$.

\vspace*{3mm}

Assume that  $\mathcal{E}(G_{S})=\mathcal{E}(G_{V(G)\backslash S})= \mathcal{E}(G)$. Then the equality in equation (\ref{eq2}) holds. Thus by Lemma \ref{l1}, there exists a real orthogonal matrix $P$ such that $PA(G)$, $P\left[A(G_{S})-\dfrac{\alpha}{n}I_{n}\right]$ and $P\left[A(G_{V(G)\backslash S})-\left(\dfrac{n-\alpha}{n}\right)I_{n}\right]$ are positive semi-definite. Moreover since $A$ is symmetric, we can further assume that the orthogonal matrix $P$ is symmetric.\\[2mm]
Now,
\begin{equation}\label{eq4}P\left[A(G_{S})-\dfrac{\alpha}{n}I_{n}\right]=PA(G)+P\left[\begin{array}{cc}
\left(1-\dfrac{\alpha}{n}\right)I_{\alpha}&\textbf{0}\\[3mm]
\textbf{0}&-\dfrac{\alpha}{n}I_{n-\alpha}
\end{array}\right].\end{equation}
Since $PA(G)$ and $P\left[A(G_{S})-\dfrac{\alpha}{n}I_{n}\right]$ are positive semi-definite, by symmetry the matrix\\ $P\left[\begin{array}{cc}
	\left(1-\dfrac{\alpha}{n}\right)I_{\alpha}&\textbf{0}\\[3mm]
	\textbf{0}&-\dfrac{\alpha}{n}I_{n-\alpha}
\end{array}\right]$ is symmetric. Let $P=\left[\begin{array}{cc}
P_{11}& P_{12}\\[3mm]
P_{21}&P_{22}
\end{array}\right]$ and the matrix $P_{11}$ be of order $\alpha$. Then by the symmetry of the matrix $P$, we get \\[2mm]   	
\begin{equation*}P\left[\begin{array}{cc}
	\left(1-\dfrac{\alpha}{n}\right)I_{\alpha}&\textbf{0}\\\\
	\textbf{0}&-\dfrac{\alpha}{n}I_{n-\alpha}
\end{array}\right]=\left[\begin{array}{cc}
\left(1-\dfrac{\alpha}{n}\right)P_{11}&-\dfrac{\alpha}{n}P_{12}\\\\
\left(1-\dfrac{\alpha}{n}\right)P_{21}&-\dfrac{\alpha}{n}P_{22}
\end{array}\right]=\left[\begin{array}{cc}
\left(1-\dfrac{\alpha}{n}\right)P_{11}&\left(1-\dfrac{\alpha}{n}\right)P_{12}\\\\-\dfrac{\alpha}{n}P_{21}&-\dfrac{\alpha}{n}P_{22}
\end{array}\right]\end{equation*}
Thus, $P_{12}=P_{21}=\textbf{0}$, otherwise $1-\dfrac{\alpha}{n}=-\dfrac{\alpha}{n}$, that is, $1=0$, a contradiction. Therefore,
$P=\left[\begin{array}{cc}
	P_{11}& \textbf{0}\\[3mm]
	\textbf{0}&P_{22}
\end{array}\right]$, $P_{11}$ and $P_{22}$ are orthogonal matrices. Henceforth,     
$$PA(G)=\left[\begin{array}{cc}
\textbf{0}&P_{11}X\\[3mm]
P_{22}X^{T}& P_{22}A(G[V(G)\backslash S])
\end{array}\right].$$ 
Since $PA(G)$ is positive semi-definite, by Lemma \ref{l2}, we must have $P_{11}X=\textbf{0}$ and so $X=\textbf{0}$ as $P_{11}$ is orthogonal, which is a contradiction. Thus $\mathcal{E}(G_{S})$ and $\mathcal{E}(G_{V(G)\backslash S})$ both cannot be equal $\mathcal{E}(G)$, and hence by (i), the proof of (ii) follows.
\end{proof}
Now we are ready to prove Conjecture \ref{c1}.\\[2mm]
\textbf{Proof of Conjecture \ref{c1}:}
Let $\emptyset \neq S\subsetneq V(G)$ and $|S|=\alpha$. If $G$ is an empty graph, then $\mathcal{E}(G)=0$ and $\mathcal{E}(G_{S})=\dfrac{2\alpha(n-\alpha)}{n}$, so $\mathcal{E}(G_S)>
\mathcal{E}(G)$. Otherwise, $G$ has a connected component $H$ of order at least 2. Let $S$ be a vertex independent set of $H$. Then by Theorem \ref{t1} (ii), either $\mathcal{E}(G_{S})>\mathcal{E}(G)$ or $\mathcal{E}(G_{V(G)\backslash S})> \mathcal{E}(G)$. This completes the proof.

\vspace*{4mm}
\subsection{Equienergetic graphs with self-loops}
	In this subsection, we construct infinite families of graphs with self loops that are equienergetic.
	Let $G$ and $H$ be two graphs. The join $G\vee H$ of $G$ and $H$ is a
	graph obtained from $G$ and $H$ by joining each vertex of $G$ to every
	vertices in $H$. We denote $n$ copies of $G$ by $nG$. The matrix whose entries are all ones and of size $p\times q$ is denoted by $J_{p\times q}$.\\[2mm]
The following lemma proved in \cite{rak} is useful to determine the spectra of join of two regular graphs.
\begin{lemma}{\rm\cite{rak}}\label{lemma1}
	For $i=1, 2$, let $M_{i}$ be a normal matrix of order $n_{i}$ having all its row sums equal to $r_{i}$.
	Suppose
	$r_{i}$, $\theta_{i2},\theta_{i3},\ldots,\theta_{in_{i}}$ are the
	eigenvalues of $M_{i}$, then for any two constants a and b, the
	eigenvalues of
	$M:=\left[\begin{array}{cc} M_{1}&aJ_{n_{1}\times n_{2}}\\
		bJ_{n_{2}\times n_{1}}&M_{2}\end{array}\right]$, are $\theta_{ij}$
	for $i=1, 2$, $j=2, 3,\ldots, n_{i}$ and the two roots of the
	quadratic equation $(x-r_{1})(x-r_{2})-abn_{1}n_{2}=0$.
\end{lemma}
\begin{center}
	\begin{figure}[h]\label{fig}
		~~~~~~~~~~~~~~~~~\includegraphics{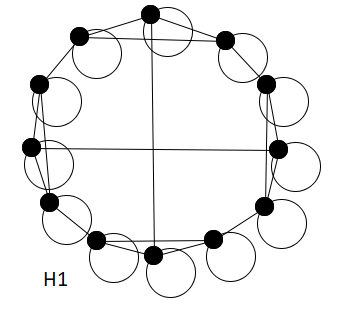}~~~~\includegraphics{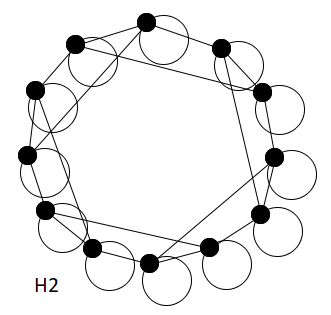}
		\caption{Graphs $H_{1}$ and $H_{2}$}
	\end{figure}
\end{center}

\begin{remark}\label{rrmk}
	Let $H_{1}$ and $H_{2}$ be the graphs as depicted in Fig. \ref{fig}. Then the spectrum of $A(H_{1})$ is $\{-2,-1,-1,0,1,1,1,1,2,3,3,4\}$ and the spectrum of $A(H_{1})$ is $\{-1,-1,-1,0,0,\\0,1,1,3,3,3,4\}$.	
\end{remark}
\begin{theorem}
	The self-loop graphs $nH_{1}\vee n\overline{K_{12}}$ and $nH_{2}\vee n\overline{K_{12}}$ are non-isomorphic and equienergetic for all integer $n\ge 1$. 	
\end{theorem}
\begin{proof}
	Let $\Gamma_{1}=nH_{1}\vee n\overline{K_{12}}$ and $\Gamma_{2}=nH_{2}\vee n\overline{K_{12}}$. Then $\Gamma_{1}$ and $\Gamma_{2}$ are non-isomorphic graphs on $24n$ vertices. From Lemma \ref{lemma1} and by Remark \ref{rrmk}, we get,\\[3mm]
		{\small$Spec(\Gamma_{1})=\left(\begin{array}{ccccccccc}
				2-2\sqrt{36n^2+1}&\pm2&-1&0&1&3&4&2+2\sqrt{36n^2+1}\\&&&&&&&&\\
				1&n&2n&13n-1&4n&2n&n-1&1
			\end{array}\right)$}\\[2mm]
		where the eigenvalues of $A(\Gamma_{1})$ are listed in the first row and the corresponding eigenvalue multiplicity is given in the second row.\\
		\noindent Also,\\[2mm]
		$Spec(\Gamma_{2})=\left(\begin{array}{ccccccccc}
			2-2\sqrt{36n^2+1}&-1&0&1&3&4&2+2\sqrt{36n^2+1}\\&&&&&&&&\\
			1&3n&15n-1&2n&3n&n-1&1	
		\end{array}\right)$.\\[2mm]
		Since the graphs $H_{1}\vee \overline{K_{12}}$ and $H_{2}\vee \overline{K_{12}}$
		has 12n vertices with self-loops, we have\\
		$$\mathcal{E}(\Gamma_{1})=\sum_{i=1}^{24n}\left|\lambda_{i}(\Gamma_{1})-\dfrac{1}{2}\right|=24n-4+4\sqrt{36n^2+1}.$$
		and 
		$$\mathcal{E}(\Gamma_{2})=\sum_{i=1}^{24n}\left|\lambda_{i}(\Gamma_{2})-\dfrac{1}{2}\right|=24n-4+4\sqrt{36n^2+1}.$$
		Thus, the graphs $\Gamma_{1}$ and $\Gamma_{2}$ are non-isomorphic and equienergetic.
	\end{proof}
	\noindent Similarly, we have the following theorem. 
	\begin{theorem}
		The self-loop graphs $nH_{1}\vee n{K_{12}}$ and $nH_{2}\vee n{K_{12}}$ are non-isomorphic and equienergetic for all integer $n\ge 1$. 	
	\end{theorem}
	The following corollary is immediate from the above two theorems.
	\begin{corollary}
		There exists a pair of non-isomorphic equienergetic graphs with self-loops on $24n$ vertices for all $n\ge 1$.	
	\end{corollary}
\section{Conclusions}
In this work, we confirm the conjecture that there exist a subset $S$ of $V(G)$ such that  $\mathcal{E}(G_{S})> \mathcal{E}(G)$ by showing that a vertex  independent set or its complement set satisfies this condition. In future,
it would be interesting to determine all possible sets $S$ for which the condition holds. Also, in this paper, pairs of connected self-loops graphs which are equienergetic are presented using join operation. The problem of determining connected graphs with self-loops whose energies are equal to the ordinary energies of the underlying graphs would be interesting for future work.  
\section*{Declaration of competing interest}
The author declare no conflict of interest.

\section*{Acknowledgements}
 K. C. Das is supported by National Research Foundation funded by the Korean government (Grant No. 2021R1F1A1050646).


\begin{thebibliography}{99}
 	\bibitem{rak}
 C. Adiga, B. R.  Rakshith, Upper bounds for the extended energy of graphs and some extended equienergetic graphs, Opuscula Math., 38 (2018), 5--13.	

\bibitem{AD1} S. Akbari, K. C. Das, M. Ghahremani, F. K.-Moftakhar, E. Raoufi, Energy of graphs containing disjoint cycles, MATCH Commun. Math. Comput. Chem. 86 (3) (2021) 543--547.

\bibitem{AD2} S. Akbari, K. C. Das, S. K. Ghezelahmad, F. K.-Moftakhar, Hypoenergetic and nonhypoenergetic digraphs, Linear Algebra Appl. 618 (2021) 129--143.

\bibitem{akbari} S. Akbari, H. A. Menderj, M. H. Ang, J. Lim, Z. C. Ng, Some results on spectrum and energy of graphs with loops, Bull. Malays. Math. Sci. Soc. 46 (2023) 94.
\bibitem{E2} M. H. Akhbari, K. K. Choong, F. Movahedi, A note on the minimum edge dominating energy of graphs. J. Appl. Math. Comput. 63 (2020) 295--310.
\bibitem{E3}
S. A. U. H.Bokhary, H. Tabassum, The energy of some tree dendrimers, J. Appl. Math. Comput. 68 (2022), 1033–1045.
\bibitem{DA1} K. C. Das, A. Alazemi, M.  An\dj{}eli\'c, On energy and Laplacian energy of chain graphs, Discrete Appl. Math. 284 (2020) 391--400.

\bibitem{so} J. Day, W. So, Singular value inequality and graph energy change, Electron. J. Linear Algebra 16 (2007) 291–-299.
 
\bibitem{gutman} I. Gutman, I. Red\v{z}epovi\'{c}, B. Furtula, A. M. Sahal, Energy of graphs with self-loops, MATCH Commun. Math. Comput. Chem. 87 (2021) 645-–652.  

\bibitem{horn} R. Horn and C. Johnson, {\it Matrix Analysis}, Cambridge University Press, Cambridge, 1989.
	
\bibitem{c2} I. Jovanovi\'{c}, E. Zogi\'{c}, E. Glogi\'{c}, On the conjecture related to the energy of graphs with self-loops, MATCH Commun. Math. Comput. Chem. 89 (2023) 479--488.
\bibitem{E1}
C. Johnson, R. Sankar, Graph energy and topological descriptors of zero divisor
graph associated with commutative ring,  J. Appl. Math. Comput. 69 (2023) 2641--2656.
\bibitem{A1}
I. Red\v{z}epovi\'{c}, B. Furtula, Predictive potential of eigenvalue-based topological molecular descriptors, J. Comput. Aided Mol. Des. 34 (2020) 975--982.
\bibitem{A2}
N. Tratnik, S. Radenkovi\'{c}, I. Red\v{z}epovi\'{c},  M. Fin\v{s}gar, P. \v{Z}igert Pleter\v{s}ek, Predicting Corrosion Inhibition Effectiveness by Molecular Descriptors of Weighted Chemical Graphs, Croat. Chem. Acta  94 (2021) 177–184.
\end{thebibliography}
\end{document}